\documentclass[12pt]{amsart}
\usepackage{dsfont,mathtools,amssymb}
\usepackage{color}

\mathtoolsset{showonlyrefs,showmanualtags}

\theoremstyle{definition}
\newtheorem{theorem}{Theorem}[section]
\newtheorem{proposition}[theorem]{Proposition}
\newtheorem{lemma}[theorem]{Lemma}
\newtheorem{corollary}[theorem]{Corollary}
\newtheorem{remark}[theorem]{Remark}

\newtheorem{assumption}[theorem]{Assumption}

\newtheorem*{acknowledgement}{Acknowledgement}
\numberwithin{equation}{section}
\numberwithin{figure}{section}

\newcommand{\eps}{\varepsilon}

\renewcommand{\rho}{\varrho}
\renewcommand{\phi}{\varphi}
\newcommand{\skp}[2]{\left\langle  #1 , #2 \right\rangle}
\newcommand{\h}{\frac{1}{2}}
\newcommand{\R}{\mathbb{R}}

\DeclareMathOperator{\var}{var}
\DeclareMathOperator{\cov}{cov}
\DeclareMathOperator{\Hess}{Hess}

\DeclareMathOperator{\osc}{osc}
\DeclareMathOperator{\Id}{Id}
\DeclareMathOperator{\diag}{diag}

\usepackage{hyperref}
\hypersetup{pdftex}

 \definecolor{darkblue}{rgb}{0,0,0.6}
 \hypersetup{
     pdftitle={A Brascamp-Lieb type covariance estimate},
     pdfauthor={Georg Menz},
     pdfsubject={MSC: primary 82B20; secondary 60K35, 82C26.},
     pdfkeywords={covariance estimate}{lattice system}{continuous spin}{decay of correlations}{Gibbs measure},
     colorlinks=true,   		
     linkcolor=darkblue,     	
     citecolor=darkblue,    	
     filecolor=darkblue,     	
     urlcolor=darkblue        	
 }

\title[A Brascamp-Lieb type covariance estimate]{A Brascamp-Lieb type covariance estimate}

\date{\today}



\author{Georg Menz}
\address{Georg Menz\\ Stanford University}
\email{gmenz@stanford.edu}

\begin{document}
\begin{abstract}
In this article, we derive a new covariance estimate. The estimate has a similar structure as the Brascamp-Lieb inequality and is optimal for ferromagnetic Gaussian measures. It can be naturally applied to deduce decay of correlations of lattice systems of continuous spins. We also discuss the relation of the new estimate with known estimates like a weighted estimate due to Helffer~\& Ledoux. The main ingredient of the proof of the new estimate is a \emph{directional Poincar\`e inequality} which seems to be unknown. 
\end{abstract}

\maketitle

\section{Introduction} 

The main goal of this article is to deduce a new covariance estimate for a certain class of Gibbs measures 
\begin{equation*}
 \mu(dx) =  \frac{1}{Z} \ \exp{\left(-H(x)\right)} \ dx,  
\end{equation*}
 on a finite-dimensional Euclidean space $X$ (see Section~\ref{s_decay_of_correlations_linear} and Theorem~\ref{p_decay_of_correlations_linear} below). Here and later on,~$Z$ denotes a generic normalization constant turning~$\mu$ into a probability measure. The covariance estimate can be seen as an analogue of the Brascamp-Lieb inequality~(BLI), which estimates variances. The BLI was originally introduced by Brascamp \& Lieb in \cite{BL}:
\begin{theorem}[Brascamp \& Lieb]\label{p_brascamp_lieb}
  Let $H:X \to \mathbb{R}$ be a smooth strictly convex function. Then for all smooth functions $f$   
\begin{equation}\label{e_brascamp_lieb}
 \var_\mu(f) :=\int \left( f - \int f \ d\mu \right)^2 d\mu \leq \int \skp{\nabla f}{\left( \Hess H \right)^{-1} \nabla f} \ d \mu.  
 \end{equation}
\end{theorem}
The main difference between the BLI  and our estimate is that
\begin{itemize}
  \item our estimate applies to covariances,
  \item it also handles non-convex Hamiltonians,
  \item in the convex case the bound is slightly weaker than in the BLI. 
\end{itemize}
The covariance estimate of Theorem~\ref{p_decay_of_correlations_linear} can be interpreted in the following way: The correlations of a non-convex perturbed Gibbs measure are dominated by the correlations of an suitable chosen Gaussian measure with ferromagnetic interaction. The proof of Theorem~\ref{p_decay_of_correlations_linear} is given in Section~\ref{s_decay_of_correlations_linear} and is based on a new type of functional inequality which we call \emph{directional Poincar\'e inequality}  (see~Theorem~\ref{p_main_proposition} below). The proof the directional Poincar\'e inequality (PI) is based on ideas which were outlined by Ledoux for the proof of the weighted covariance estimate (cf.~\cite{L} and Theorem~\ref{p_weighted_covariance_estimate}). 
\medskip

The use of the new covariance estimate is illustrated in Section~\ref{s_decay_of_correlations}, where we show how the estimate can be used to deduce decay of correlations of certain lattice systems of continuous spins. We distinguish two cases:\medskip

In Section~\ref{s_sub_exponential_decay_corr} we consider exponential decay of correlations. We show that the new covariance estimate yields a well-known weighted covariance estimate due to Helffer (see Theorem~\ref{p_weighted_covariance_estimate}, \cite[Section 4]{He2} or \cite[Proposition 2.1 or 3.1]{L}). This weighted covariance estimate is the central ingredient in a common method to deduce exponential decay of correlations for unbounded spin systems with a non-convex single-site potential and a weak finite-range interaction (see \cite[Theorem 2.1]{He2}, \cite[Theorem 1.1]{B-H1}, \cite[Theorem 3.1]{B-H2} or \cite[Proposition 6.2]{L}). Additionally, we show how Theorem~\ref{p_decay_of_correlations_linear} directly yields an exponential decay of correlations in this situation without relying on Theorem~\ref{p_weighted_covariance_estimate} (see Corollary \ref{p_decay_of_correlations_example} and Proposition \ref{p_application_of_decay_of_covariances}).\medskip

In Section~\ref{s_sub_algebraic_decay_corr} we consider algebraic decay of correlations. Using the new Brascamp-Lieb type covariance estimate, we give a criterion to deduce algebraic decay of correlations of lattice systems of continuous spins (see Proposition~\ref{p_algebraic_decay_correlations}). \medskip

The main result of this article (i.e.~Theorem~\ref{p_decay_of_correlations_linear}) was successfully applied in other articles of the author: Because there is a deep connection between decay of correlations and the validity of certain functional inequalities like the logarithmic Sobolev inequality (LSI) or the PI (see for example \cite{Ze1,Zeg96,He2,B-H1,Yos99,Yos01} or \cite{B-H1} for an overview), it is not surprising that Theorem~\ref{p_decay_of_correlations_linear} is one of the key ingredients to derive the LSI for the canonical ensemble $\mu_{N,m}$ in the case of a weak two-body interaction\cite{Menz11}. Additionally, Proposition~\ref{p_algebraic_decay_correlations} was used in~\cite{OR_rev} to refine the Otto-Reznikoff approach to the LSI.\medskip

We conclude the introduction by making a comment on the origin of the
content of this article. Most of the material of this article is
contained in the dissertation\cite{Diss} of the author but
unpublished until now. The proof of the Brascamp-Lieb type covariance estimate of Theorem~\ref{p_decay_of_correlations_linear} emerged out of joint discussions with Felix Otto.

 \medskip

\section{The Brascamp-Lieb type covariance estimate and its proof.} \label{s_decay_of_correlations_linear}

We consider a finite dimensional Euclidean space $X$. Norms $|\cdot|$ and gradients $\nabla$ are derived from the Euclidean structure. If a probability measure $\mu$ on $X$ satisfies the PI, we directly obtain the following standard covariance estimate:
 \begin{lemma}\label{p_poincare_covariance}
   Assume $\mu$ satisfies PI with constant $\varrho$. Then for any smooth function $f$ and $g$ we have 
   \begin{align}
     | \cov_{\mu}(f,g) | \leq \frac{1}{\varrho} \left( \int |\nabla f|^2 \ d \mu \right)^{\h} \left( \int |\nabla g|^2 \ d \mu \right)^{\h}.  \label{e_basic_covariance_estimate}
   \end{align}
 \end{lemma}
Even if the estimate \eqref{e_basic_covariance_estimate} is optimal (cf.~\cite[Remark~4]{OR07}), it does not yield information about the dependence of the covariance on the specific coordinates. Hence, the estimate~\eqref{e_basic_covariance_estimate} is useless for deducing decay of covariances. For example, let us consider a Gaussian Gibbs measure 
\begin{equation*}
  \mu (dx) = \frac{1}{Z} \exp \left( - x \cdot Ax \right) \ dx
 \end{equation*} 
on $\R^N$ with a symmetric and positive definite $N \times N$- Matrix $A$. Then it is known that
\begin{equation}\label{e_gaussian_covariance}
 \cov_\mu (x_n,x_k) = \left(A^{-1} \right)_{nk} \leq \frac{1}{\varrho}.
\end{equation}
Therefore, we can hope for a finer estimate than~\eqref{e_basic_covariance_estimate} that is also sensitive to the dependence of the functions $f$ and $g$ on the specific coordinates $x_i$. Our covariance estimate shows this feature:
 \begin{assumption}
   \label{p_assumption_on_growth}
   We assume that the Hamiltonian $H$ of the Gibbs measure $\mu$ is convex at infinity i.e. $H$ is a bounded perturbation of a convex function. It follows from the observation by Bobkov \cite{Bob} -- all log-concave measures satisfy PI -- and the perturbation lemma of Holley-Stroock~\cite{HS} (cf.~Theorem~\ref{local:thm:HolleyStroock}) that $\mu$ satisfies PI with a unspecified constant $\tilde \varrho>0$.
\end{assumption}

\begin{theorem}[Covariance estimate, Otto \& Menz] \label{p_decay_of_correlations_linear} 
  We consider a probability measure $d\mu:= Z^{-1} \exp (-H(x)) \ dx$ on a direct product of Euclidean spaces $X= X_1 \times \cdots \times X_N$. We assume that
\begin{itemize}
\item the conditional measures $\mu(dx_i | \bar x_i )$, $1\leq i \leq N$, satisfy a uniform PI with constant $\varrho_i>0$ which means that for all smooth functions $f: X_i \to \mathbb{R}$ 
\begin{equation*}
   \var_{\mu} (f): = \int \left(f - \int f d \mu \right)^2 d \mu \leq \frac{1}{\varrho_i} \int |\nabla f|^2 d\mu
\end{equation*}
uniformly in~$\bar x_i$.
\item the numbers $\kappa_{ij}$, $1 \leq i \neq j \leq N$, satisfy
   \begin{equation*}
|\nabla_i \nabla_j H(x)|\leq \kappa_{ij}  < \infty     
   \end{equation*}
uniformly in $x \in X$. Here, $|\cdot|$ denotes the operator norm of a bilinear form. 
 \item the symmetric matrix $A=(A_{ij})_{N \times N}$ defined by
  \begin{equation} \label{e_definition_of_A}
A_{ij} =
\begin{cases}
  \varrho_i, & \mbox{if }\;  i=j , \\
  -\kappa_{ij}, & \mbox{if } \; i< j,
\end{cases}
  \end{equation}
is positive definite. 
\end{itemize}
 Then for all smooth functions $f$ and $g$ 
   \begin{equation}
     \label{e_covariance estimate}
      | \cov_{\mu}(f,g) | \leq  \sum_{i,j=1}^N \left( A^{-1} \right)_{ij} \ \left( \int |\nabla_i f|^2 \ d \mu \right)^{\h} \left( \int |\nabla_j g|^2 \ d \mu \right)^{\h}.
   \end{equation}
\end{theorem}

 The structure of the estimate in Theorem \ref{p_decay_of_correlations_linear} is related to the BLI in the sense that variance is replaced by covariance and that $\Hess H$ is replaced by $A$. 
\begin{remark}[Connection to BLI] \label{p_connection_to_BLI}  We assume $X_i = \R$ for $i \in \{1, \ldots, N \}$ and let $A$ be a symmetric positive definite $N \times N$- matrix. We consider a ferromagnetic Gaussian  Hamiltonian given by
\begin{equation*}
 H(x)  = \frac{1}{2} \sum_{1\leq i,j \leq N } x_i A_{ij} x_j \ + \sum_{1 \leq i \leq N} b_i x_i, \qquad A_{ij}, b_j \in \mathbb{R},
  \end{equation*} 
where ferromagnetic means that the coupling is attractive i.e. 
\[
 A_{ij}=A_{ji} \leq 0 \quad \mbox{for } i< j \in \{1, \ldots, N  \}. 
\]

Then the covariance estimate \eqref{e_covariance estimate} coincides with the BLI given by~\eqref{e_brascamp_lieb} provided the function $f=g$ is an affine function.
\end{remark}
The next remark considers the optimality of Theorem \ref{p_decay_of_correlations_linear}.
\begin{remark}[Optimality]\label{r_optimality}  Provided the Hamiltonian H is ferromagnetic Gaussian, the estimate of Theorem \ref{p_decay_of_correlations_linear} is optimal. This remark is verified by setting $f(x_n)=x_n$ and $g(x_k)=x_k$ and using \eqref{e_gaussian_covariance}.
\end{remark}

\begin{remark}[Criterion for PI]\label{p_OR_for_SG} Theorem \ref{p_decay_of_correlations_linear} contains a well-known criterion for PI i.e.  
If $A \geq \varrho \Id$, $\varrho>0$, then~$\mu$ satisfies a PI with constant~$\varrho$, which means that for all smooth functions $f$
\begin{equation*}
   \var_{\mu} (f): = \int \left(f - \int f d \mu \right)^2 d \mu \leq \frac{1}{\varrho} \int |\nabla f|^2 d\mu. \tag{PI}
\end{equation*}

\begin{equation*}
    A \geq \varrho \Id, \quad \varrho>0  \qquad \Rightarrow \qquad \mu \; \mbox{satisfies PI with constant }\varrho.
  \end{equation*}
\end{remark}

The assumption under which Theorem \ref{p_decay_of_correlations_linear} holds has the same algebraic structure as the assumption in the Otto-Reznikoff criterion for LSI (cf.~\cite[Theorem 1]{OR07}). The only difference is that the uniform LSI constant for the single-site conditional measures is replaced by the uniform PI constant.  \smallskip

Starting point of the proof of Theorem \ref{p_decay_of_correlations_linear} is a representation of the covariance, which was used by Helffer \cite{He6} to give another proof of the BLI. More precisely, one can express the covariance of the measure $\mu$ as   
\begin{equation}
  \label{e_representation_covariance}
  \cov_{\mu}(f,g) = \int \nabla \varphi \cdot \nabla g \ d \mu,
\end{equation}
where the potential $\varphi$ is defined as the solution of the elliptic equation 
\begin{equation}
  \label{e_definition_of_varphi}
 - \nabla \cdot \left( \mu  \nabla \varphi\right) = \left(f - \int f \ d \mu \right)  \mu . 
\end{equation}
Here we used the convention, that $\mu$ also denotes the Lebesgue density of the probability measure $\mu$. As a solution of \eqref{e_definition_of_varphi} we understand any $\varphi\in H^1(\mu)$ such that for all $\zeta \in H^1(\mu)$
\begin{equation}
  \label{e_sense_of_solution}
  \int \nabla \zeta \cdot \nabla \varphi \ d \mu = \int \zeta \left(f - \int f \ d\mu \right) \ d \mu. 
\end{equation}
The existence of such solutions follows directly from the Riesz representation theorem applied to 
\begin{equation} \label{e_Hilbert_H}
\mathcal{H} = H^1(\mu) \cap \left\{\varphi , \ \int \varphi d \mu =0  \right\}  
\end{equation}
equipped with the inner product 
\begin{equation}
  \label{e_inner_product_on_H}
 \int \nabla \zeta \cdot \nabla \varphi \ d \mu. 
\end{equation}
The completeness of $\mathcal{H}$ w.r.t.~the chosen inner product follows from the fact that $\mu$ satisfies some PI, which is guaranteed by our Assumption~\ref{p_assumption_on_growth}.\medskip

Let us return to the proof of Theorem \ref{p_decay_of_correlations_linear}. An application of the Cauchy-Schwarz inequality to \eqref{e_representation_covariance} yields
\begin{equation}
  \label{e_cauchy_schwarz_covariance}
  \left| \cov_{\mu}(f,g) \right| \leq \sum_{i=1}^N \left( \int |\nabla_i \varphi |^2 d \mu \right)^\h \left( \int |\nabla_i g |^2 d \mu \right)^\h  .
\end{equation}
Now, an application of the following theorem yields the desired estimate~\eqref{e_covariance estimate} and completes the proof of Theorem~\ref{p_decay_of_correlations_linear}.


\begin{theorem}[Directional PI] \label{p_main_proposition} 
 Assume that the conditions of Theorem \ref{p_decay_of_correlations_linear} are satisfied. 
For any function $f$ let the potential $\varphi$ be a solution of \eqref{e_definition_of_varphi}.  
 Then for all $ i \in \{ 1, \ldots, N \}$
\begin{equation}\label{e_main_proposition}
 \left( \int |\nabla_i \varphi|^2  d \mu \right)^{\frac{1}{2}}  \leq \sum_{j=1}^{N} \left(A^{-1} \right)_{ij} \left(\int |\nabla_j f|^2 d \mu \right)^{\frac{1}{2}}. 
\end{equation} 
\end{theorem}

Before we turn to the proof of Theorem~\ref{p_main_proposition}, let us explain why we call the estimate~\eqref{e_main_proposition} directional PI. For this let us recall the dual formulation of the PI (cf.~for example~\cite{OV}), which is an easy consequence of the dual characterization of the norm on the Hilbertspace~$\mathcal{H}$ given by~\eqref{e_Hilbert_H} and~\eqref{e_inner_product_on_H}.
\begin{lemma}[Dual formulation of the PI]
  A probability measure $\mu$ satisfies PI with constant $\varrho>0$ if and only if for any function $f$ and the solution $\varphi$ of \eqref{e_definition_of_varphi}
  \begin{equation}\label{e_weak_SG}
 \left( \int |\nabla \varphi|^2 \ d\mu \right)^{\frac{1}{2}}  \leq \frac{1}{\varrho} \left(\int |\nabla f|^2 d\mu \right)^{\frac{1}{2}}.
\end{equation}  
\end{lemma}

Note that the directional PI given by~\eqref{e_main_proposition} estimates each coordinate of the gradient of~$\varphi$ separately and therefore is a refinement of the dual formulation of the PI given by~\eqref{e_weak_SG}. As in~\cite[Section~3]{OV}, function $\varphi$ formally denotes the tangent vector at of the curve~$(1+\varepsilon f) \mu$ at~$\varepsilon=0$. Therefore, $\nabla \varphi$ can be interpreted as the infinitesimal optimal displacement transporting the measure $\mu$ into $(1+\varepsilon f) \mu$ (cf.~\cite[Section~5]{OV}). So, the left hand side of \eqref{e_main_proposition} measures the average flux of mass into the direction of the $i$-th coordinate against a weighted gradient of $f$. For this reason we call \eqref{e_main_proposition} directional PI. \medskip

 One can also interpret the estimate \eqref{e_main_proposition} in terms of the Witten complex (for a nice overview see~\cite{He_Book}). At least formally one can introduce the inverse Witten-Laplacian $A_1^{-1}$ as
\begin{equation*}
 A_1^{-1} \ \nabla f := \nabla \varphi,  
\end{equation*}
which maps the gradient of some function $f$ onto the gradient of the solution $\varphi$ of the equation \eqref{e_definition_of_varphi}. Let $\Pi_i$ denote the projection onto the space $X_i$, $i \in \{1, \ldots, N \} $. Then the estimate \eqref{e_main_proposition} becomes a weighted estimate of the $L^2$-operator norm of $ \Pi_i A_1^{-1}$. \medskip

Let us now turn to the proof of Theorem \ref{p_main_proposition}, which is the only missing ingredient in the proof of Theorem~\ref{p_decay_of_correlations_linear}. The argument is very basic. It combines the core inequality of Ledoux's argument for \cite[Proposition 3.1]{L} with linear algebra that was used in the argument of \cite[Theorem 1]{OR07}.

\begin{proof}[Proof of Theorem \ref{p_main_proposition}]
  To make the main ideas of the argument more visible, we assume that the Euclidean spaces $X_i$, $i \in \left\{ 1, \ldots, N \right\}$, are one dimensional i.e.~$X_i= \R$. The argument for general Euclidean spaces $X_i$ is almost the same. Then the product space $X=X_1 \times \cdots \times X_N$ becomes $\R^N$. The gradient $\nabla_i$ on $X_i$ is just the partial derivative  $\partial_i$ w.r.t.~the $i$-th coordinate. The first ingredient of the proof is the basic estimate for $j \in \left\{ 1, \ldots , N \right\}$
  \begin{equation}
    \label{e_e_ss_sg}
    \int \left( | \partial_j \partial_j  \varphi |^2 + \partial_j \varphi \ \partial_j \partial_j H  \ \partial_j \varphi \right) \mu (dx_j | \bar x_j)  \geq \varrho_j \int | \partial_j \varphi |^2 \mu (dx_j | \bar x_j),
  \end{equation}
which is just an equivalent formulation of the PI with constant $\varrho_j$ for the single-site measure $\mu (dx_j | \bar x_j)$ (cf.~\cite[Proposition~1.3, (1.8)]{L} or \cite{HeS,He1}). The second ingredient of the proof is the identity 
\begin{equation}
  \label{e_core_identity}
  \int \partial_j \varphi \ \partial_j f d \mu = \int \sum_{k=1}^N \left(|  \partial_j \partial_k \varphi |^2 + \partial_j \varphi \ \partial_j \partial_k H  \ \partial_k \varphi  \right) d \mu.
\end{equation}
Indeed, by partial integration one sees that
\begin{equation*}
   \int \partial_j \varphi \ \partial_j f d \mu = - \int \partial_j \partial_j \varphi \ \left( f- \int f d \mu \right)  d \mu + \int  \partial_j \varphi \ \partial_j H \left( f- \int f d \mu \right)  d \mu.
\end{equation*}
Applying now~\eqref{e_sense_of_solution} on the terms of the r.h.s.~yields the identity
\begin{align*}
   \int \partial_j \varphi \ \partial_j f \ d \mu & =  - \int \sum_{k=1}^N \partial_k \partial_j \partial_j \varphi \ \partial_k \varphi  \ d \mu +  \int \sum_{k=1}^N \partial_k \partial_j \varphi \ \partial_j H \ \partial_k \varphi  \ d \mu \\
   & \qquad + \int \sum_{k=1}^N  \partial_j \varphi \ \partial_k \partial_j H \ \partial_k \varphi   \ d \mu.
\end{align*}
Let us have a closer look at the second term on the r.h.s~of the last identity. It follows from the definition of $\mu$ that
\begin{align*}
  \int \sum_{k=1}^N \partial_k \partial_j \varphi \ \partial_j H \ \partial_k \varphi  \ d \mu &=  - \frac{1}{Z} \int \sum_{k=1}^N \partial_k \partial_j \varphi(x) \  \ \partial_k \varphi(x) \  \partial_j \exp\left(- H(x) \right) \ dx \\
  & =   \int \sum_{k=1}^N \partial_j \partial_k \partial_j \varphi \ \ \partial_k \varphi  \ d \mu  +   \int \sum_{k=1}^N \partial_k \partial_j \varphi \  \partial_j \partial_k \varphi  \ d \mu 
\end{align*}
A combination of the last two formulas yields the desired identity~\eqref{e_core_identity}. \smallskip

Now, we turn to the proof of~\eqref{e_main_proposition}. A combination of~\eqref{e_e_ss_sg} and~\eqref{e_core_identity} yields the estimate
\begin{align*}
  \int \partial_j \varphi \ \partial_j f \ d \mu & \geq \varrho_j \int | \partial_j \varphi |^2 d\mu  + \int \sum_{k=1, \ k \neq j}^N  \partial_j \varphi \ \partial_j \partial_k H  \ \partial_k \varphi  \ d \mu \\
  & \geq  \varrho_j \int | \partial_j \varphi |^2 d \mu  - \sum_{k=1, \ k \neq j}^N \kappa_{jk} \int  \partial_j \varphi \ \partial_k \varphi \  d \mu.
\end{align*}
Applying Cauchy-Schwarz on the last estimate yields for all $j \in \left\{ 1, \ldots, N \right\}$
\begin{align}
   \left( \int | \partial_j f | ^2 d \mu \right)^{\h} & \geq \varrho_j \left( \int | \partial_j \varphi |^2  d \mu \right)^{\h}  - \sum_{k=1, \ k \neq j}^N \kappa_{jk} \left(  \int  | \partial_k \varphi |^2 d \mu \right)^{\h} \notag \\
 & = \sum_{k=1}^N A_{jk} \left(  \int  | \partial_k \varphi |^2 d \mu \right)^{\h} .  \label{e_basic_estimate_cs}
\end{align}
A simple linear algebra argument outlined in \cite[Lemma~9]{OR07} shows that the elements of the inverse of $A$ are non negative i.e.~$\left( A^{-1} \right)_{ij} \geq 0$ for all $i, j \in \left\{ 1, \ldots, N \right\}$. Hence, \eqref{e_basic_estimate_cs} yields
\begin{align*}
  \sum_{j=1}^N \left( A^{-1} \right)_{ij} \left( \int | \partial_j f | ^2 d \mu \right)^{\h} & \geq   \sum_{j=1}^N \left( A^{-1} \right)_{ij}  \sum_{k=1}^N A_{jk} \left(  \int  | \partial_k \varphi |^2 d \mu \right)^{\h} \\
  & = \delta_{ik} \left(  \int  | \partial_k \varphi |^2 d \mu \right)^{\h}  = \left(  \int  | \partial_i \varphi |^2 d \mu \right)^{\h}.
\end{align*}
\end{proof}
The proof of Theorem \ref{p_decay_of_correlations_linear} is just a direct application of Theorem \ref{p_main_proposition}.
\begin{proof}[Proof of Theorem \ref{p_decay_of_correlations_linear}]
Using the definition of $\varphi$, cf.~\eqref{e_definition_of_varphi}, we obtain the following estimate of the covariance
\begin{align}
  \cov_{\mu} \left( f,g \right) & = \int f \left( g - \int g \ \mu \right) d \mu \notag \\
  & = \int \nabla \varphi \cdot \nabla g \ d \mu \notag \\
  & \leq \sum_{j=1}^N \left(\int |\nabla_j \varphi|^2 d \mu \right)^{\frac{1}{2}} \left(\int |\nabla_j g|^2 d \mu \right)^{\frac{1}{2}} \notag 
\end{align}
Now, the statement follows directly from Theorem \ref{p_main_proposition}.  
\end{proof}

\section{Application of the B-L type covariance estimate: Decay of correlations}\label{s_decay_of_correlations}
In this section we show how Theorem~\ref{p_decay_of_correlations_linear} can be used to deduce decay of correlations. We distinguish between two cases:
\begin{itemize}
\item exponential decay of correlations (see Section~\ref{s_sub_exponential_decay_corr})
\item and algebraic decay of correlations (see Section~\ref{s_sub_algebraic_decay_corr}).
\end{itemize}

\subsection{Exponential decay of correlations.}\label{s_sub_exponential_decay_corr}
 
We start with reflecting a method based on Helffer~\cite{He2} that has often been used to derive exponential decay of correlations of spin systems with finite-range interaction or exponentially decaying (cf.~\cite{B-H1} and~\cite{B-H2}). This method is based on a weighted covariance estimate, which we present in the spirit of Ledoux \cite[Proposition 3.1]{L}, but rephrase the estimate in our framework.
\begin{theorem}[Helffer, Ledoux]\label{p_weighted_covariance_estimate}
   We assume that the conditions of Theorem \ref{p_decay_of_correlations_linear} are satisfied. Additionally, we consider positive weights $d_i > 0$, $i \in \{1, \ldots N \}$. Let the diagonal $N \times N$- matrix $D$ be defined as
   \begin{equation*}
  D := \diag (d_1 \ldots, d_N).   
   \end{equation*}
  We assume that there exists $\varrho>0$ such that in the sense of quadratic forms
\begin{equation}
  \label{e_positivity_of_DAD_inverse}
   D A D^{-1} \geq \varrho \Id.  
\end{equation}
Then the matrix $A$ is positive definite and for all functions $f$ and $g$, 
   \begin{equation}
     \label{e_weighted_covariance_estiamte}
       \cov_{\mu}(f,g)  \leq  \frac{1}{ \varrho}  \left( \int |D \nabla f|^2 \ d \mu \right)^{\h} \left( \int |D^{-1}\nabla g|^2 \ d \mu \right)^{\h}.
   \end{equation}
\end{theorem}
At the end of this section, we give a new proof of Theorem~\ref{p_weighted_covariance_estimate} showing that the weighted covariance estimate~\eqref{e_weighted_covariance_estiamte} is an easy consequence of our covariance estimate of Theorem~\ref{p_decay_of_correlations_linear}. This shows that the statement of Theorem~\ref{p_decay_of_correlations_linear} is consistent with the existing literature.  
\begin{remark}\label{p_OR_for_SG_relaxed}
 Using a direct argument for deducing of Theorem \ref{p_weighted_covariance_estimate}, one sees that the condition \eqref{e_positivity_of_DAD_inverse} can be relaxed to a weaker condition (for the argument we refer the reader to~\cite[Section~1.2.1]{Diss} or~\cite[Proposition 3.2]{MFC}). More precisely, let the symmetric $N \times N $-matrix $\mathcal{A}(x)=(\mathcal{A}_{ij}(x)) $ be defined by
\begin{equation}\label{e_definition_of_mathcal_A}
\mathcal{A}_{ij} (x) =
\begin{cases}
  \varrho_i, & \mbox{if } \; i=j , \\
  \nabla_i \nabla_j H(x) , & \mbox{if } \; i < j.
\end{cases}
  \end{equation}
 Assume that there is $\varrho>0$ such that for all $x \in X$ 
\begin{equation}
  \label{e_positivity_of_D_mathcal_A_D_inverse}
   D \mathcal{A}(x) D^{-1} \geq \varrho \Id.  
\end{equation}
\end{remark}

Now, let us explain how the weighted covariance estimate of Theorem~\ref{p_weighted_covariance_estimate} can be used to deduce exponential decay of correlations. Let us consider a metric $\delta(\cdot,\cdot)$ on the set of sites $\{1, \ldots , N  \}$ of the spin system. For an arbitrary but fixed site $l \in \{1, \ldots , N  \}$ one chooses  
\begin{equation*}
 d_i:= \exp \left(- \delta(i,l) \right) 
\end{equation*}
as weights in Theorem \ref{p_weighted_covariance_estimate}. Because the triangle inequality implies
\begin{equation*}
 \frac{d_i}{d_j} = \exp \left(\delta(j,l) - \delta(i,l) \right) \leq \exp \left(\delta(j,i) \right), 
\end{equation*}
 a direct application of Theorem \ref{p_weighted_covariance_estimate} yields the following criterion for exponential decay of correlations.
 \begin{corollary}[Helffer \& Ledoux]\label{p_decay_of_correlations_example}
   Assume that the conditions of Theorem \ref{p_decay_of_correlations_linear} are satisfied. Additionally, we consider a metric $\delta(\cdot,\cdot)$ on the set $\{1, \ldots, N \}$ and the symmetric $N \times N$- matrix $\tilde A = (\tilde A_{ij})$ defined by
\begin{equation}\label{e_definition_of_tilde_A}
  \tilde A_{ij} =
\begin{cases}
  \varrho_i, & \mbox{if } \; i=j , \\
  - \exp \left( \delta(i,j) \right) \kappa_{ij}, & \mbox{if } \; i < j.
\end{cases}
  \end{equation}
   We assume that there exists $\tilde \varrho>0$ such that in the sense of quadratic forms
\begin{equation}
  \label{e_positivity_of_tilde_A}
   \tilde A \geq \tilde \varrho \Id.  
\end{equation}
Then for all functions $f=f(x_i)$ and $g=g(x_j)$, $i, j \in \{1, \ldots, N \}$, 
   \begin{equation}
     \label{e_covariance_decay}
      | \cov_{\mu}(f,g) | \leq  \frac{1}{\tilde \varrho} \ \exp \left(- \delta(i,j) \right)  \left( \int |\nabla_i f|^2 \ d \mu \right)^{\h} \left( \int |\nabla_j g |^2 \ d \mu \right)^{\h}.
   \end{equation}
\end{corollary}
This criterion may also be stated more generally for functions with arbitrary disjoint supports. It is implicitly contained in the prelude of~\cite[Proposition~6.2]{L}. \medskip

At the end of this section we will also give a direct proof of Corollary~\ref{p_decay_of_correlations_example}, which is just based on the covariance estimate of Theorem~\ref{p_decay_of_correlations_linear} and does not need the weighted covariance estimate of Theorem~\ref{p_weighted_covariance_estimate}. \medskip

Now, let us give an example how Corollary~\ref{p_decay_of_correlations_example} can be applied. For that purpose we consider a two-dimensional lattice system with non-convex single-site potential and weak nearest-neighbor interaction. The same type of argument would also work for any dimension and finite-range interaction. Let $X$ denote a two-dimensional periodic lattice of $N$-sites and let $\delta (\cdot,\cdot)$ denote the graph distance on it. We assume that $\mu \in \mathcal{P}(X)$ has the Hamiltonian 
\begin{equation}
  \label{e_example_hamiltonian}
  H(x) = \sum_{i} \psi (x_i) - \varepsilon \sum_{\delta (i,j)=1} x_i x_j,
\end{equation}
where the smooth potential $\psi$ is a bounded perturbation of a Gaussian in the sense that
\begin{equation*}
 \psi(x) = \frac{1}{2} x^2 + \delta \psi (x) \qquad \qquad \mbox{and} \qquad   \qquad \sup_{\R} |\delta \psi (x)| < \infty.  
\end{equation*}
By a combination of the Bakry-\'Emery criterion (cf.~Theorem~\ref{local:thm:BakryEmery}) and the of Holley-Stroock perturbation principle (cf.~Theorem~\ref{local:thm:HolleyStroock}) all conditional measures $\mu(dx_i| \bar x_i)$ satisfy a uniform LSI with constant $\Delta := \exp \left( - \osc \delta \psi  \right)$. From \eqref{e_example_hamiltonian} we see that 
\begin{equation*}
\kappa_{ij} = \sup_{x} |\nabla_i \nabla_j H(x) | = \varepsilon.  
\end{equation*}
Hence, we know that if the interaction is sufficiently weak in the sense of $\varepsilon < \frac{\Delta}{4}$, the matrix $A$ of Theorem \ref{p_decay_of_correlations_linear} satisfies 
\begin{equation*}
A \geq \left( \Delta - 4 \varepsilon \right) \Id.  
\end{equation*}
Analogously one obtains that if $\varepsilon < \frac{\Delta}{4} e^{-1}$, the matrix $\tilde A$ of Corollary \ref{p_decay_of_correlations_example} satisfies 
\begin{equation*}
\tilde A \geq \left( \Delta - 4 \varepsilon  e \right) \Id.  
\end{equation*}
Therefore, an application of Corollary \ref{p_decay_of_correlations_example} yields exponential decay of correlations:
\begin{proposition}\label{p_application_of_decay_of_covariances}
  Assume that $ \varepsilon < \frac{\Delta}{4} e^{-1}$. Then for any functions $f=f(x_i)$ and $g=g(x_j)$, $i, j \in \{1, \ldots, N \}$, 
  \begin{equation*}
 | \cov_{\mu}(f,g) | \leq  \frac{1}{\Delta - 4 \varepsilon  e  } \ \exp \left(- \delta(i,j) \right)  \left( \int |\nabla_i f|^2 \ d \mu \right)^{\h} \left( \int |\nabla_j g|^2 \ d \mu \right)^{\h}.   
  \end{equation*}
\end{proposition}
This statement reproduces the correlation bounds established by Helffer \cite{He2} and reproved by Ledoux in \cite[Proposition 6.2]{L}. \medskip

Let us now prove the statements mentioned in this section.

\begin{proof}[Proof of Theorem \ref{p_weighted_covariance_estimate} using Theorem \ref{p_decay_of_correlations_linear}]
 We start with deducing that $A$ is positive definite. Because $A$ is a symmetric Matrix, it suffices to show that every eigenvalue of $A$ is positive. Let $\lambda \in \R$ be an eigenvalue of $A$ with eigenvector $x$ i.e.~
  \begin{equation*}
    Ax= \lambda x.
  \end{equation*}
An application of \eqref{e_positivity_of_DAD_inverse} to the vector $Dx$ yields
\begin{equation*}
  \lambda | D x|^2 = Dx \cdot DAx = Dx \cdot DA D^{-1} Dx \geq \varrho |D x^2| >0,
\end{equation*}
which implies $\lambda >0$.\newline
Now, we will deduce \eqref{e_weighted_covariance_estiamte}. Because $A$ is symmetric, the inverse $A^{-1}$ also is symmetric. Therefore, an application of Theorem \ref{p_decay_of_correlations_linear} yields the estimate
  \begin{align*}
    \cov_{\mu}(f,g) & \leq \sum_{i,j=1}^N \left(A^{-1} \right)_{ij}  \ \left( \int |\nabla_i f|^2 \ d \mu \right)^{\h} \left( \int |\nabla_j g|^2 \ d \mu \right)^{\h} \\
    & = \sum_{i,j=1}^N d_j \left(A^{-1}\right)_{ji}  d_i^{-1}  \ \left( \int |d_i \nabla_i f|^2 \ d \mu \right)^{\h} \left( \int |d_j^{-1} \nabla_j g|^2 \ d \mu \right)^{\h} \\
    & = DA^{-1}D^{-1} z \cdot \tilde z \\
    & \leq |DA^{-1}D^{-1} z| \ |\tilde z|,
  \end{align*}
  where the vectors $z, \tilde z \in \R^N$ are defined for $i,j \in \{1, \ldots, N \}$ by
  \begin{equation*}
z_i:= \left( \int |d_i \nabla_i f|^2 \ d \mu \right)^{\h} \qquad \mbox{and} \qquad \tilde z_j :=\left( \int |d_j^{-1} \nabla_j g|^2 \ d \mu \right)^{\h}.
  \end{equation*}
  Therefore, \eqref{e_weighted_covariance_estiamte} is verified provided 
  \begin{equation}
    \label{e_weighted_A_estimate}
    | D A^{-1} D^{-1}  z | \leq \frac{1}{\varrho} \ |z|
  \end{equation}
  holds for any $z \in \R^N$. From the hypothesis \eqref{e_positivity_of_DAD_inverse} it follows that
  \begin{align*}
    \varrho \  z \cdot z & \leq D A D^{-1} z \cdot z    \\
    & \leq  |D A D^{-1} z | \ |z|.
  \end{align*}
 Hence, we have  
 \begin{equation*}
 | z | \leq \frac{1}{\varrho} \ |D A D^{-1} z |,  
 \end{equation*}
which immediately yields \eqref{e_weighted_A_estimate}.
\end{proof}

\begin{proof}[Direct proof of Corollary~ \ref{p_decay_of_correlations_example} using only Theorem~\ref{p_decay_of_correlations_linear}]
  Let us fix two indices $i,j \in \{1, \ldots, N \}$. Let $f$ and $g$ be arbitrary functions just depending on $x_i$ and $x_j$ respectively. We apply Theorem \ref{p_decay_of_correlations_linear} and get
  \begin{equation}
    \label{e_estimation_covariance_corollary}
 \cov_{\mu} (f,g) \leq \left( A^{-1} \right)_{ij}  \left( \int |\nabla_i f|^2 \ d \mu \right)^{\h} \left( \int | \nabla_j g|^2 \ d \mu \right)^{\h},   
  \end{equation}
  where $A$ is defined as in \eqref{e_definition_of_A}. Therefore, it remains to estimate the element $\left( A^{-1} \right)_{ij}$. By Neumann series (also called the random walk expansion of $A^{-1}$ (cf. \cite{BFS})) we have
\begin{align}
  \left( A^{-1} \right)_{ij} & = \delta_{ij} \frac{1}{\varrho_i} + \frac{\kappa_{ij}}{\varrho_i \varrho_j} + \sum_{s=1}^N \frac{\kappa_{is} \kappa_{sj}}{\varrho_i \varrho_s \varrho_j} + \sum_{s,l =1}^N \frac{\kappa_{is} \kappa_{s l} \kappa_{l j}}{\varrho_i \varrho_{s} \varrho_{l} \varrho_j} + \cdots \cdots \notag \\
  & = \delta_{ij} \frac{1}{\varrho_i} +  \frac{e^{- \delta (i,j)}}{e^{- \delta (i,j)}} \frac{\kappa_{ij}}{\varrho_i \varrho_j} + \sum_{s=1}^N  \frac{e^{- \delta (i,s)}e^{- \delta (s,j)}}{e^{- \delta (i,s)}e^{- \delta (s,j)}} \frac{\kappa_{is} \kappa_{sj}}{\varrho_i \varrho_s \varrho_j}  \notag \\
  &  \quad  + \sum_{s,l =1}^N \frac{e^{- \delta (i,s)}e^{- \delta (s,l)}e^{- \delta (l,j)}}{e^{- \delta (i,s)}e^{- \delta (s,l)}e^{- \delta (l,j)}}  \frac{\kappa_{is} \kappa_{s l} \kappa_{l j}}{\varrho_i \varrho_{s} \varrho_{l} \varrho_j} + \cdots \cdots. \label{e_estimation_inverse_element_step_0} 
\end{align}
By the triangle inequality we get
\begin{equation*}
e^{- \delta (i,s)}e^{- \delta (s,j)} \leq e^{- \delta (i,j)}  
\end{equation*}
for all $i,s,j \in \{1, \ldots, N \}$. Hence, we can continue the estimation of \eqref{e_estimation_inverse_element_step_0} as
\begin{align}
  \left( A^{-1} \right)_{ij} & \leq e^{-  \delta (i,j)} \left(\tilde A^{-1} \right)_{ij}, \label{e_estimation_inverse_element_step_1}
\end{align}
where $\tilde A$ is defined as in \eqref{e_definition_of_tilde_A}. By \eqref{e_positivity_of_tilde_A} we have the bound
\begin{equation*}
 \left(\tilde A^{-1} \right)_{ij} \leq \frac{1}{\tilde \varrho},  
\end{equation*}
which together with \eqref{e_estimation_covariance_corollary} and \eqref{e_estimation_inverse_element_step_1} finishes the proof.
\end{proof}

\subsection{Algebraic decay of correlations}\label{s_sub_algebraic_decay_corr}

In this section we show how Theorem~\ref{p_decay_of_correlations_linear} can be used to deduce an algebraic decay of correlations in the case of algebraically decaying interaction. Because in the article~\cite{OR_rev} the statement of Proposition~\ref{p_algebraic_decay_correlations} is applied to a $d$-dimensional lattice system, we change the notation a little bit.   
\begin{proposition}\label{p_algebraic_decay_correlations} 
Let $\Lambda \subset \mathds{Z}^d$ an arbitrary finite subset of the $d$-dimensional lattice $\mathds{Z}^d$. We consider a probability measure $d\mu:= Z^{-1} \exp (-H(x)) \ dx$ on $\mathds{R}^\Lambda$. We assume that
\begin{itemize}
\item the conditional measures $\mu(dx_i | \bar x_i )$, $i \in \Lambda $, satisfy a uniform PI with constant $\varrho_i>0$.
\item the numbers $\kappa_{ij}$, $i \neq j, i,j \in \Lambda$, satisfy
   \begin{equation*}
|\nabla_i \nabla_j H(x)|\leq \kappa_{ij}  < \infty     
   \end{equation*}
uniformly in $x \in \mathds{R}^\Lambda$. Here, $|\cdot|$ denotes the operator norm of a bilinear form. 
\item the numbers $\kappa_{ij}$ decay algebraically in the sense of
  \begin{align}
    \label{e_algeb_decay_of_kappa}
    \kappa_{ij} \lesssim \frac{1}{|i-j|^{d+\alpha} +1} 
  \end{align}
for some $\alpha>0$.
 \item the symmetric matrix $A=(A_{ij})_{N \times N}$ defined by
  \begin{equation*} 
A_{ij} =
\begin{cases}
  \varrho_i, & \mbox{if }\;  i=j , \\
  -\kappa_{ij}, & \mbox{if } \; i< j,
\end{cases}
  \end{equation*}
is strictly diagonally dominant  i.e.~for some $\delta > 0$ it holds for any $i \in \Lambda$ 
\begin{equation}\label{e_strictly_diag_dominant_A}
  \sum_{j \in \Lambda, j \neq i} |A_{ij}| + \delta \le A_{ii}. 
\end{equation}
\end{itemize}
Then for all functions $f=f(x_i)$ and $g=g(x_j)$, $i, j \in \Lambda$, 
   \begin{equation}
     \label{e_covariance_decay_algebraic}
      | \cov_{\mu}(f,g) | \lesssim (A^{-1})_{ij}   \left( \int |\nabla_i f|^2 \ d \mu \right)^{\h} \left( \int |\nabla_j g |^2 \ d \mu \right)^{\h}
   \end{equation}
and for any $i, j \in \Lambda$
\begin{align}
      \label{e_decay_M_inverse}
      |(A^{-1})_{ij}| \lesssim \frac{1}{|i-j|^{d + \tilde \alpha}+1},
    \end{align}
    for some $\tilde \alpha >0$.
  \end{proposition}
 
\begin{proof}[Proof of Proposition~\ref{p_algebraic_decay_correlations}]
  Because the matrix $A$ is strictly diagonal dominant in the sense of~\eqref{e_strictly_diag_dominant_A} by assumption, the matrix $A$ is also positive definite. Therefore an application of Theorem~\ref{p_decay_of_correlations_linear} directly yields the estimate~\eqref{e_covariance_decay_algebraic}. So, it is only left to deduce the estimate~\eqref{e_decay_M_inverse}. As in the proof of Corollary~\ref{p_decay_of_correlations_example} the Neumann series representation of $A^{-1}$ yields for $i \neq j$
\begin{align}
  \left( A^{-1} \right)_{ij} & = \underbrace{\frac{\kappa_{ij}}{\varrho_i \varrho_j}}_{=: T_0} + \underbrace{\sum_{s\in \Lambda} \frac{\kappa_{is} \kappa_{sj}}{\varrho_i \varrho_s \varrho_j}}_{=:T_1} + \underbrace{\sum_{s_1, s_2 \in \Lambda} \frac{\kappa_{is_1} \kappa_{s_1 s_2} \kappa_{s_2 j}}{\varrho_i \varrho_{s_1} \varrho_{s_2} \varrho_j}}_{=:T_2} + \cdots \cdots  \label{e_represent_A_inverse} \\
& = \sum_{k=0}^\infty T_k.   \notag
\end{align}
It follows from our assumption~\eqref{e_strictly_diag_dominant_A} that
\begin{align}
  \frac{\kappa_{\tilde m n}}{\varrho_n}  \leq \sum_{m \in \Lambda} \frac{\kappa_{nm}}{\varrho_n} \leq  c <1 \quad \mbox{uniformly in }  n, \tilde m \in \Lambda. \label{e_neumann_bound_diag_dominant}
\end{align}
Therefore we get the estimate
\begin{align*}
  T_k \leq c^k .
\end{align*}
Let $\tilde n$ denote the smallest integer larger than $\frac{\log |i-j|^{d + \alpha} }{|\log c|}$. Then we have
\begin{align}
  \sum_{k=\tilde n}^\infty T_k  \leq c^{\tilde n} \sum_{k=0}^\infty c^k  \leq \frac{1}{|i-j|^{d+\alpha}} C. \label{e_est_A_inverse_first_part}
\end{align}
Considering~\eqref{e_represent_A_inverse} it only remains to estimate~$\sum_{k=0}^{\tilde n} T_k$. Assume for the moment that 
\begin{align}\label{e_est_T_k_A_inverse}
  T_k \leq C \ \frac{(k+1)^{d+\alpha +1 }}{|i-j|^{d+\alpha}}
\end{align}
uniform in $k \in \mathbb{N}$. Then we get the estimate
\begin{align}
\label{e_est_A_inverse_second_part}  \sum_{k=0}^{\tilde n} T_k &\leq  C  \frac{{(\tilde n +1)}^{d+ \alpha +1}}{|i-j|^{d+\alpha}} \\
& \leq C \frac{(\log |i-j|^{d + \alpha} +1 )^{d+\alpha+1}}{|i-j|^{d+\alpha}} \leq C \frac{1}{|i-j|^{d+\frac{\alpha}{2}}}. \notag
\end{align}
A combination of \eqref{e_represent_A_inverse},~\eqref{e_est_A_inverse_first_part}, and~\eqref{e_est_A_inverse_second_part} yields the desired statement~\eqref{e_decay_M_inverse}. \medskip

In order to complete the argument we have to the estimate~\eqref{e_est_T_k_A_inverse}. Consider the multi-indexes $i, s_1, \ldots s_k,j \in \Lambda \subset \mathds{Z}^d$. For convenience we set $s_0 =i$ and $s_{k+1}=j$. Let $\tilde n$ be the integer such that 
\begin{align*}
|i_{\tilde n} - j_{\tilde n}|  = \max (|i_l - j_l| \ l \in \left\{1, \ldots, d \right\} ) .
\end{align*}
Then there is at least one pair of $(s_0,s_1)$, $(s_1,s_2))$, $\ldots$, $(s_{k-1}, s_k)$, or $s_k, s_{k+1}$ that satisfies the estimate
\begin{align*}
  |(s_l)_{\tilde n} - (s_{l+1})_{\tilde n}| \geq \frac{1}{k+1} |i_{\tilde n}-j_{\tilde n}|.
\end{align*}
By the equivalence of norms in finite-dimensional vector-spaces the last inequality yields
\begin{align} \label{e_est_one_has_to_satisfy}
  |s_l - s_{l+1}| \geq C \frac{1}{k+1} |i-j|.
\end{align}
Therefore we have 
\begin{align*}
  T_k & = \sum_{s_1, \ldots, s_k \in \Lambda} \frac{\kappa_{s_0 s_1} \kappa_{s_1 s_2} \ldots  \kappa_{s_{k} s_{k+1}}}{\varrho_i \varrho_{s_1} \ldots \varrho_{k} \varrho_j} \\
  & \leq \sum_{\substack{s_1, \ldots, s_k \in \Lambda \\ (s_0,s_1) \mbox{ \tiny satisfies }\eqref{e_est_one_has_to_satisfy}}} \frac{\kappa_{s_0 s_1} \kappa_{s_1 s_2} \ldots  \kappa_{s_{k} s_{k+1}}}{\varrho_i \varrho_{s_1} \ldots \varrho_{k} \varrho_j} \\
  & \qquad + \sum_{\substack{s_1, \ldots, s_k \in \Lambda \\ (s_1,s_2) \mbox{ \tiny satisfies }\eqref{e_est_one_has_to_satisfy}}} \frac{\kappa_{s_0 s_1} \kappa_{s_1 s_2} \ldots  \kappa_{s_{k} s_{k+1}}}{\varrho_i \varrho_{s_1} \ldots \varrho_{k} \varrho_j} \\
& \qquad + \ldots \ +  \sum_{\substack{s_1, \ldots, s_k \in \Lambda \\ (s_k,s_{k+1}) \mbox{ \tiny satisfies }\eqref{e_est_one_has_to_satisfy}}} \frac{\kappa_{s_0 s_1} \kappa_{s_1 s_2} \ldots  \kappa_{s_{k} s_{k+1}}}{\varrho_i \varrho_{s_1} \ldots \varrho_{k} \varrho_j} .
\end{align*}
We show how the second term on the right hand side can be estimated. The estimation of the other terms works almost the same, hence we skip it. We have
\begin{align*}
&  \sum_{\substack{s_1, \ldots, s_k \in \Lambda \\ (s_1,s_2) \mbox{ \tiny satisfies }\eqref{e_est_one_has_to_satisfy}}} \frac{\kappa_{s_0 s_1} \kappa_{s_1 s_2} \ldots  \kappa_{s_{k} s_{k+1}}}{\varrho_i \varrho_{s_1} \ldots \varrho_{k} \varrho_j} \\
 & \ \overset{~\eqref{e_algeb_decay_of_kappa}}{\leq}  C \sum_{\substack{s_1, \ldots, s_k \in \Lambda \\ (s_1,s_2) \mbox{ \tiny satisfies }\eqref{e_est_one_has_to_satisfy}}}  \frac{1}{|s_1 - s_2|^{d+\alpha} +1} \ \frac{\kappa_{s_0 s_1} \kappa_{s_2 s_3} \ldots  \kappa_{s_{k} s_{k+1}}}{\varrho_i \varrho_{s_1} \ldots \varrho_{k} \varrho_j} \\
& \  \overset{\eqref{e_est_one_has_to_satisfy}}{\leq} C  \frac{(k+1)^{d + \alpha}}{|i - j|^{d+\alpha} +1} \sum_{s_1, \ldots, s_k \in \Lambda }   \ \frac{\kappa_{s_0 s_1} \kappa_{s_2 s_3} \ldots  \kappa_{s_{k} s_{k+1}}}{\varrho_i \varrho_{s_1} \ldots \varrho_{k} \varrho_j} \\
& \  \overset{\eqref{e_neumann_bound_diag_dominant}}{\leq} C  \frac{(k+1)^{d + \alpha}}{|i - j|^{d+\alpha} +1}.
\end{align*}
With similar bounds for the other terms we get the desired estimate
\begin{align*}
    T_k \leq C  \frac{(k+1)^{d + \alpha +1}}{|i - j|^{d+\alpha} +1},
\end{align*}
which closes the argument.
\end{proof}

\appendix
\section{The criterion of Bakry-\'Emery and the Holley-Stroock perturbation principle}\label{s_BE_HS}

In this section we state the criterion Bakry-\'Emery and the Holley-Stroock perturbation principle, which we used in the main part of this article to deduce the PI for certain measures. Because we only work with the PI in this article we state those criteria for the PI. However, note that both criteria also hold on the stronger level of the LSI. The \emph{Bakry-\'Emery criterion} connects convexity of the Hamiltonian to the validity of the~PI.
\begin{theorem}[Bakry-\'Emery criterion {\cite[Proposition 3, Corollaire 2]{BE85}}]\label{local:thm:BakryEmery}
 Let $H: D \to \mathbb{R}$ be a Hamiltonian with Gibbs measure $$\mu(dx)=Z_\mu^{-1} \exp\left( -\varepsilon^{-1} H(x) \right) \ dx$$ on a convex domain $D$ and assume that $\nabla^2 H(x) \geq \lambda >0$ for all $x \in \mathbb{R}^n$. Then $\mu$ satisfies PI with constant $\varrho$ satisfying
 \begin{equation}
  \varrho \geq \frac{\lambda}{\eps} . 
 \end{equation}
\end{theorem}
In non-convex cases the standard tool to deduce the PI is the \emph{Holley-Stroock perturbation principle}.
\begin{theorem}[Holley-Stroock perturbation principle {\cite[p. 1184]{HS}}]\label{local:thm:HolleyStroock}
  Let $H$ be a Hamiltonian with Gibbs measure $$\mu(dx)=Z_\mu^{-1} \exp\left( -\eps^{-1} H(x) \right) \ dx.$$ Further, let $\tilde H$ denote a bounded perturbation of $H$ and let $\tilde \mu_\varepsilon$ denote the Gibbs measure associated to the Hamiltonian $\tilde H$.
 If $\mu$ satisfies PI with constant $\varrho$ then also $\tilde\mu$ satisfies the PI with constant$\tilde \varrho$, where the constants satisfies the bound
 \begin{equation}\label{local:e:HolleyStroockPI-LSI}
\tilde \varrho \geq \exp \left( -\varepsilon^{-1} \osc (H- \tilde H)\right) \varrho,
 \end{equation}
 where $\osc (H - \tilde H) := \sup (H - \tilde H) - \inf (H - \tilde H)$.
\end{theorem}
The perturbation principle of Holley-Stroock~\cite{HS} allows to deduce the PI constants of non-convex Hamiltonian from the PI of an appropriately convexified Hamiltonian. However due to its perturbative nature, the dependence of the PI constant~$\tilde \varrho$ usually is bad in physical parameters like system size or temperature.

\begin{acknowledgement}
The author wants to thank Felix Otto for working with him and finding
out a simple proof of
Theorem~\ref{p_decay_of_correlations_linear}. Additionally, the author
wants to thank Maria Westdickenberg (ne\'e Reznikoff) and Christian
Loeschcke for the fruitful and inspiring discussions on this
topic. The author was financially supported by the Deutsche
Forschungsgemeinschaft through the Gottfried Wilhelm Leibniz program
and partially by the Bonn International Graduate School in Mathematics during the years 2007 to 2009, where most of the content of this article originated.   
\end{acknowledgement}

\bibliographystyle{amsalpha}
\bibliography{b-l_covariance_estimate.bib}

\end{document}